\documentclass{amsart}
\usepackage{amsmath, amssymb,epic,graphicx,mathrsfs,enumerate}
\usepackage[all]{xy}

\usepackage{mathtools}
\usepackage{bbm, dsfont}
\usepackage{amsthm}
\usepackage{amssymb}
\usepackage{latexsym}
\usepackage{longtable}
\usepackage{epsfig}
\usepackage{amsmath}
\usepackage{hhline}
\usepackage{url}
\usepackage{enumitem}

\DeclareMathOperator{\GU}{GU}
\DeclareMathOperator{\Sp}{Sp}

\DeclareMathOperator{\SL}{SL}
\DeclareMathOperator{\SU}{SU}
\DeclareMathOperator{\SO}{SO}
\DeclareMathOperator{\GL}{GL} 
 
 \DeclareMathOperator{\soc}{soc}
\DeclareMathOperator{\aut}{Aut} \DeclareMathOperator{\out}{Out}

\DeclareMathOperator{\alt}{Alt}

\newcommand{\mn}{\widetilde{m}}

\newcommand{\w}{\widetilde}

\newcommand{\FF}{\mathbb F}
\newcommand{\N}{\mathbb N}

\renewcommand{\emptyset}{\varnothing}
\newcommand{\Cl}{\mathcal C}

\newtheorem{theorem}{Theorem}
\newtheorem{corollary}[theorem]{Corollary}

 \newtheorem{lemma}[theorem]{Lemma}
\newtheorem{proposition}[theorem]{Proposition}

\numberwithin{equation}{section}

\renewcommand{\footnote}{\endnote}
\newcommand{\ignore}[1]{}\makeglossary

\begin{document}
\bibliographystyle{amsplain}

\title[The Chebotarev invariant]{The Chebotarev invariant for  direct products of nonabelian finite simple groups}
\author{Jessica Anzanello}
\address{J. Anzanello, Universit\`a di Padova, Dipartimento di Matematica ``Tullio Levi-Civita'', Via Trieste 63, 35121 Padova, Italy}
\email{jessica.anzanello@studenti.unipd.it}
\author{Andrea Lucchini}
\address{A. Lucchini, Universit\`a di Padova, Dipartimento di Matematica ``Tullio Levi-Civita'', Via Trieste 63, 35121 Padova, Italy}
\email{lucchini@math.unipd.it}
\author{Gareth Tracey}
\address{G. Tracey, Mathematics Institute , University of Warwick, Coventry, CV4 7AL}
\email{Gareth.Tracey@warwick.ac.uk}


\begin{abstract}
A subset $\{g_1, \ldots , g_d\}$ of a finite group $G$  invariably generates $G$ if
$\{g_1^{x_1}, \ldots , g_d^{x_d}\}$ generates $G$ for every choice of $x_i \in G$. The Chebotarev invariant $C(G)$ of $G$ is the expected value of the random variable $n$ that is minimal
subject to the requirement that $n$ randomly chosen elements of $G$ invariably generate $G$. In this paper, we show that if $G$ is a nonabelian finite simple group, then $C(G)$ is absolutely bounded. More generally, we show that if $G$ is a direct product of $k$ nonabelian finite simple groups, then $C(G)=\log{k}/\log{\alpha(G)}+O(1)$, where $\alpha$ is an invariant completely determined by the proportion of derangements of the primitive permutation actions of the factors in $G$. It follows from the proof of the Boston-Shalev conjecture that $C(G)=O(\log{k})$. We also derive sharp bounds on the expected number of generators for $G$. 
\end{abstract}

\maketitle

\section{Introduction}

In this paper, we are interested in \textit{generation} and \textit{invariable generation} of finite groups. While the first has a long and rich history in finite group theory, the latter is still quite new and unexplored. Invariable generation was introduced in the early nineties by Dixon, with motivation from computational Galois theory (see \cite{Dix92}). Following his work, we say that a subset $\{g_1,\dots,g_t\}$ of a finite group $G$ \textit{invariably generates} $G$ if $\{g{_1}^{x_1},\dots,g{_t}^{x_t}\}$ generates $G$ for every $t$-tuple $(x_1,\dots,x_t) \in G^t$.
Our motivating questions are the following:
{\sl{if we choose random elements from a finite group $G$ independently, with replacement, and with the uniform distribution, how many should we expect to pick until the elements chosen generate $G$? And how many to invariably generate $G$?}}

We can formalise the quantities needed to answer the question above in the following way. Let $G$ be a non-trivial finite group and let $x=(x_n)_{n \in \N}$ be a sequence of independent, uniformly distributed, $G$-valued random variables. 
We may define two random variables (two waiting times) by:

\begin{enumerate}
	\item   $\tau_{G} = \min\{n \geq 1: \langle x_1, \dots,x_n \rangle=G\} \in [1,+\infty]$;
	
	\item  $\tau_{I,G} = \min\{n \geq 1: \{x_1, \dots, x_n\} \text{ invariably generates } G\} \in [1,+\infty]$.
\end{enumerate}

We denote the expectations of $\tau_G$ and $\tau_{\text{I},G}$ respectively with $e_1(G)$ and $C(G)$. The latter is known as the \textit{Chebotarev invariant} of $G$, and was firstly introduced by Kowalski and Zywina in 2012 \cite{KZ}, taking its name from its relation with the \textit{Chebotarev density theorem}.

Given a finite group $G$, one may ask how differently $e_1(G)$ and $C(G)$ behave: clearly, we have $e_1(G) \le C(G)$ and, if $G$ is abelian, $C(G)=e_1(G)$. Moreover, Kantor, Lubotzky and Shalev \cite{KanLubSha} showed that a finite group is nilpotent if and only if every generating
set is an invariable generating set. But the difference $C(G)-e_1(G)$ is not small in general. For example, Kowalski and Zywina have proved the following:
let $q$ be a prime power, and consider the affine group $G_q \coloneqq \mathbb{F}_q \rtimes \mathbb{F}_q^{*}$, then $C(G_q)\sim \sqrt{|G_q|}$ as $q \rightarrow \infty$.
On the other hand, Lubotzky \cite{Lub03} and, independently, Detomi and Lucchini \cite{LuDe}, settling a conjecture by Pak, have proved that, for any finite group $G$, 
$e_1(G)=d(G)+O(\log \log|G|)$ (here $d(G)$ denotes the smallest cardinality of a generating set of $G$).
It follows in particular that, even in the metabelian case, $C(G)$ and $e_1(G)$ behave quite differently.

To conclude this overview, we remark that the behaviour of $C(G_q)$, where $G_q$ is the affine group of order $q(q-1)$, led Kowalski and Zywina to conjecture that $C(G)=O(\sqrt{|G|})$ for every finite group $G$. 
Progress on the conjecture was first made by Kantor, Lubotzky and Shalev in \cite{KanLubSha}, where it was shown that $C(G)=O(\sqrt{|G|}\log|G|)$, and the conjecture was confirmed by Lucchini in \cite[Theorem 1]{LuccCheb}:
{\sl{there exists an absolute constant $\beta$ such that $C(G) \leq \beta \sqrt{|G|}$ whenever $G$ is a finite group}}. Moreover, in \cite{LuccGarCheb}, Lucchini and Tracey showed that $\beta \le 5/3$ whenever $G$ is soluble, and that this is best possible, with equality if and only if $G=C_2 \times C_2$. Furthermore, in the same paper, they showed that, in general, for each $\epsilon > 0$, there exists a constant $c_{\epsilon}$ such that $C(G) \le (1+\epsilon) \sqrt{|G|}+c_{\epsilon}$.

In this paper, we derive upper bounds for $e_1(G)$ and for $C(G)$, where $G$ is a direct product of $k$ nonabelian finite simple groups. In both cases, we give a linear bound for these expectations in terms of the logarithm of the number of direct factors. Moreover, we prove that, if $S$ is a finite nonabelian simple group, then $C(S) \le \gamma_1$, for an absolute constant $\gamma_1$. An analogous result for $e_1(S)$, with a precise estimate of the absolute constant, was obtained by the second author in \cite[Theorem 7]{e1Alt6}.

The result for $e_1(G)$ could be deduced from \cite[Theorem 20]{LuDe}, but we prefer to give an alternative direct proof, which also allows to obtain a more precise statement.

\begin{theorem}\label{maine}
	Let $G=T_1 \times \dots \times T_k$ be a direct product of $k \ge 2$ nonabelian finite simple groups. Then
	\begin{equation*}
		e_1(G) \leq \max_{1 \leq i \leq k}{\frac{\log{k}}{\log{l(T_i)}}}+6,
	\end{equation*}
	where $l(T_i)$ denotes the minimum index of a proper subgroup of $T_i$.
\end{theorem}
	Here and throughout this paper, the symbol $\log$ denotes the logarithm to the base $2$. In Section 4, we will show that for all finite nonabelian simple groups $T$, we have $e_1(T^k) = \max_{1 \leq i \leq k}{\frac{\log{k}}{\log{l(T_i)}}}+c(k)$, where $-3\le c(k)\le 6$. Thus, Theorem \ref{maine} is close to best possible.

The main result of the paper is that a similar statement holds for invariable generation, in which the role of $l(T_i)$  is played by an invariant from the theory of derangements. Indeed, for a finite simple group $S$ and an $S$-set $\Omega$, we denote by $\delta_{\Omega}(S)$ the proportion of derangements in $S^{\Omega}$. (Here, and throughout the paper, $S^{\Omega}$ will denote the image of the induced action of $S$ on $\Omega$). We write $\delta(S)$ for the minimal value of $\delta_{\Omega}(S)$ as $\Omega$ runs over all primitive $S$-sets, and we define $\alpha(S):=(1-\delta(S))^{-1}$. Alternatively, denoting with $\tilde M$ the union of the $S$-conjugates of $M,$ $\alpha(S)$ is the the minimum of $|S|/|\w{M}|$ as $\w{M}$ ranges over the maximal subgroups of $S$. The Boston--Shalev conjecture, proved by Fulman and Guralnick \cite{FGC2C3}, states that there exists an absolute constant $\alpha_0>0$ such that $\alpha(S)>1+\alpha_0$ for all nonabelian finite simple groups $S$. Our result can now be stated as follows.  
\begin{theorem}\label{maineinv}
	Let $G=T_1 \times \dots \times T_k$ be a direct product of $k \ge 2$ nonabelian finite simple groups. Then
	\begin{equation*}
		C(G) \leq \max_{1 \leq i \leq k}{\frac{\log{k}}{\log{\alpha(T_i)}}}+\gamma,
	\end{equation*}
	for some absolute constant $\gamma$.
\end{theorem}

We will prove in Section 4 that for all $0<\epsilon<1$ there exists an absolute constant $C_{\epsilon}>0$ such that $C(T^k)\geq (1-\epsilon){\frac{\log{k}}{\log{\alpha(T)}}}-C_{\epsilon}$ for all nonabelian finite simple groups $T$ and all $k\geq 1$. Thus, analogously to Theorem \ref{maine}, the bound in Theorem \ref{maineinv} is close to best possible. This also proves that the invariants $C(T^k)$ and $\alpha(T)$ are inextricably linked.  

Our main corollary of Theorem \ref{maineinv} is as follows, and is an immediate consequence of the work of Fulman and Guralnick mentioned above.
\begin{corollary}\label{tceb}
There exist absolute constants $\gamma_1$ and $\gamma_2$ such that
\begin{enumerate}
    \item[\upshape(i)] $C(S) \le \gamma_1$ whenever $S$ is a nonabelian finite simple group; and
    \item[\upshape(ii)] $C(G) \leq \gamma_2 \log{k}$ whenever $G=T_1 \times \cdots \times T_k $ is a direct product of $k \ge 2$ nonabelian finite simple groups $T_i$.
\end{enumerate}
\end{corollary}
The proof of Theorem \ref{maineinv} is much more difficult then the one of Theorem \ref{maine} and strongly depends on the classification of the finite nonabelian simple groups.

In the last section of the paper we will consider the particular case when $G=T^k$ is the direct power of $k$ copies of a fixed nonabelian finite simple group $T$. The study of this case will show that the estimations given in Theorems \ref{maine} and \ref{maineinv}  are close to best possible, as mentioned above. We will also prove that $C(A_5^k)-e_1(A_5^k)$ tends to infinity with $k$.

\section{An upper bound for {$e_1(G)$}{}}\label{secondsec}

Denote by $m_n(G)$ the number of maximal subgroups of $G$ with index $n$ and let
\begin{equation*}
	\mathcal{M}(G)=\sup_{n \geq 2} \frac{\log{m_n(G)}}{\log{n}}.
\end{equation*}
The invariant $\mathcal{M}(G)$ can be seen as the polynomial degree of the rate of growth of $m_n(G)$. It is roughly equal to $e_1(G)$, and precisely we have (see \cite[Theorem 1.1]{MG:LuccMosc}):
\begin{equation}\label{amp}\lceil \mathcal{M}(G)\rceil-4 \leq e_{1}(G) \leq \lceil \mathcal{M}(G) \rceil + 3.
	\end{equation}
This invariant has been studied for finite and profinite groups by various authors, see for example \cite{MG:LuccMosc}, and the references therein. 
\paragraph{}
Our aim  is to give an estimate of $\mathcal{M}(G)$ for 
\begin{equation}
	\label{G}
	G \cong \prod_{1 \leq i \leq r}T_{i}^{k_i},
\end{equation}
with $T_i$ a nonabelian simple group and $T_i \ncong T_j$ for all $i \neq j$.
Thanks to Goursat's lemma, we have a complete description of the maximal subgroups of $G$. Indeed, the maximal subgroups $M$ of $G$ are all of the form:
	\begin{equation*}
		M=T_1^{k_1} \times \dots \times T_{i-1}^{k_i-1}\times M_i\times T_{i+1}^{k_i+1}\times \dots \times T_{r}^{k_r},
	\end{equation*}
	with $i \in \{1\dots r\}$ and $M_i$ a maximal subgroup of $T_i^{k_i}$, and for $M_i$ we have two possibilities.
	\begin{enumerate}
		\item {\sl{Product type}}: $M_i= T_i\times \dots \times K_i\times \dots \times T_i$, with $K_i$ a maximal subgroup of $T_i$. In this case, $|M_i|=|T_i|^{k_i-1}|K_i|$ and thus $[G:M]=[T_i^{k_i}:M_i]=[T_i:K_i]$.
		\item  {\sl{Diagonal type}}: $M_i=\{(t_1,\cdots,t_j,\cdots,t_k,\cdots,t_{k_i}) \in T_i^{k_i}\mid t_k=t_j^{\phi}\}$,  given $j<k$ and $\phi \in$ Aut($T_i$).
		In this case, $|M_i|=|T_i|^{k_i-1}$ and thus $[G:M]=[T_i^{k_i}:M_i]=|T_i|$.
	\end{enumerate}
	Notice that, by \cite[Theorem 7]{e1Alt6}, if $S$ is a  finite nonabelian simple group, then $$e_1(S)\leq e_1(\alt(6)) \approx 2.494.$$ Hence, it follows from (\ref{amp}) that
$\mathcal{M}(S)  \leq 6$ for every 	
finite nonabelian simple group $S.$
We are now ready for the main result of this section.
\begin{proposition}
	\label{thm1}
	If $G$ is a direct product of nonabelian finite simple groups as in (\ref{G}) and $k=\sum_{i=1}^{r}k_i$, then
	\begin{equation}
		\mathcal{M}(G) \leq \max_{1 \leq i \leq r}{\frac{\log{k}}{\log{l(T_i)}}}+2,
	\end{equation}
	where $l(T_i)$ denotes the minimum index of a proper subgroup of $T_i$.
\end{proposition}
	
	\begin{proof}
		Fix $n\in\mathbb{N}$, and let $\Lambda_n:=\{i\mid m_n(T_i)\geq 1\}$, $\Sigma_n:=\{i\mid |T_i|=n\}$. Note that $\Lambda_n\cap \Sigma_n=\emptyset$. We will use the above description of the maximal subgroups of $G$ to derive upper bounds on $m_n(G)$.
        If $n<\min_{1\le i\le r} l(T_i)$, then $m_n(G)=0$, so assume that $n\geq \min_{1\le i\le r} l(T_i)$. 
  
  Suppose first that $n<\min_{1\le i\le r} l(T_i)^2$. As noted in \cite[Section 5.2, page 178]{KL}, it follows from the classification theorem that every finite nonabelian simple group has a proper subgroup of order at least $\sqrt{|T|}$ (one can also check this by inspecting the known values of $l(T)$ given by \cite{KL,V1,V2,V3,AtlasOn}).
        Thus, $l(T_i)\le \sqrt{|T_i|}$ for all nonabelian finite simple groups $T_i$. Since $n<\min_{1\le i\le r} l(T_i)^2$, we deduce that  $\Sigma_n=\emptyset$. Thus, $m_n(G)=\sum_{i\in\Lambda_n}k_i m_n(T_i)\le kn^2$, where we used that $m_n(T_i)\le n^2$ (see \cite[Theorem 1.3]{Lub03} or \cite[Theorem 21]{LuDe2}).

    Finally, assume $n\geq \min_{1\le i\le r} l(T_i)^2$. Then
		\begin{equation*}
			\begin{split}
				m_n(G)&=\sum_{i\in\Lambda_n}k_i m_n(T_i)+\sum_{i\in\Sigma_n}\binom{k_i}{2}|\text{Aut}(T_i)| \\
                &\le \sum_{i\in\Lambda_n}k_in^2 +\sum_{i\in\Sigma_n}\binom{k_i}{2}n^2 \le k^2n^2\
			\end{split}
		\end{equation*}
		since  $|\aut(T_i)| \leq |T_i|^2$ by the classification of finite simple groups. In particular
		\begin{equation*} 
			\frac{\log{m_n(G)}}{\log{n}} \leq \frac{2\log{k}}{\log{n}}+2\le \frac{2\log{k}}{\min_{1\le i\le r} \log{(l(T_i)^2)}}+2=\frac{\log{k}}{\min_{1\le i\le r} \log{l(T_i)}}+2.
		\end{equation*}
The result follows.
\end{proof}

\begin{proof}[Proof of  Theorem \ref{maine}]It follows immediately combining the previous proposition with (\ref{amp}).
\end{proof}

\section{An upper bound for {${C}(G)$}{}}

In this section, we consider the case of \textit{invariable generation}. For every maximal subgroup $M$ of a group $G$, let $\w{M}=\cup_{g \in G} M^g$ denote the union of the $G$-conjugates of $M$. Clearly, $\w{M}_1=\w{M}_2$ if the maximal subgroups $M_1$ and $M_2$ are conjugate in $G$.
Let $\mathcal{M}:=\mathcal{M}(G)$ be a set of representatives of conjugacy classes of maximal subgroups of G. The following lemma is straightforward.

\begin{lemma}
	\label{inv}
	A subset $\{g_1, \dots, g_t\}$ of a finite group $G$ invariably generates $G$ if and only if $\{g_1, \dots, g_t\} \nsubseteq \w{M}$ for all $M \in \mathcal{M}(G)$.
\end{lemma}

By definition, 
\begin{equation*}
	C(G)= \sum_{t=0}^{\infty}\mathbb{P}(\tau_{I,G} > t)=
	\sum_{t=0}^{\infty}(1-\mathbb{P}_{I}(G,t)),
\end{equation*}
where $\mathbb{P}_{I}(G,t)$ denotes the probability that $t$ randomly chosen elements of $G$ invariably generate $G$.
As noted in Lemma \ref{inv}, a subset $\{g_1, \dots, g_t\}$ fails to invariably generate $G$ if and only if it is contained in $\w{M}$, for some maximal subgroup $M \in \mathcal{M}$. Hence,
\begin{equation*}
	1-\mathbb{P}_{I}(G,t) \leq \sum_{M \in \mathcal{M}} \left( \frac{|\widetilde M|}{|G|} \right)^{t}.
\end{equation*}
For a non-negative integer $n$, let $\mn_{n}(G)$ denote the number of conjugacy classes of maximal subgroups $M$ of $G$ satisfying $\lfloor |G|/|\widetilde{M}|\rfloor=n$. Our strategy for proving Theorem \ref{maineinv} is to prove that $\mn_{n}(G)$ is polynomial in $n$ whenever $G$ is a nonabelian finite simple group. 
Our claim is trivial for sporadic groups, so we need only prove that $\mn_{n}(G)$ has a polynomial bound in $n$ when $G$ is an alternating, classical or exceptional simple group.

\subsection{Almost simple groups with alternating socle}

We start by dealing with the alternating case. For this purpose, we require a theorem of Eberhard, Ford and Green \cite{EbForGr}. Following their notation, we define the constant 
$$\delta:=1-\frac{1+\ln{\ln{2}}}{\ln{2}} \approx 0.08607.$$

Their result is as follows.

\begin{theorem}[{\cite[Theorem 1]{EbForGr}}]
	
	\label{EFGTheorem} 
	For positive integers $r$ and $k$ with $1\le k\le r/2$, let $i(r,k)$ denote the proportion of elements of the symmetric group $S_{r}$ which fix a $k$-set (set-wise). There are absolute constants $D_{1}$ and $D_{2}$ such that 
	$$D_{2}k^{-\delta}(1+\log{k})^{-3/2}\le i(r,k)\le D_{1}k^{-\delta}(1+\log{k})^{-3/2}.$$
\end{theorem}

\begin{proposition}\label{AlternatingCase} Let $G$ be an almost simple group with alternating socle. Then there exists an absolute constant $C_1$
	such that $\mn_{n}(G) \leq C_1 n^{1/\delta}$ for all $n\geq 1$.
\end{proposition}

\begin{proof} Write $\soc{(G)}=A_{r}$.
	Clearly we may assume that $r$ is as
	large as we wish. We will also assume that $G=S_{r}$, the case $G=A_{r}$ being similar.

	Fix $n\ge 1$, and let $\Cl_{intrans}(G)$ [resp. $\Cl_{imprim}(G)$, $\Cl_{prim}(G)$] be a set of representatives for the $G$-conjugacy classes of maximal subgroups $M$ of $G$ which are intransitive [resp. imprimitive, primitive], and satisfy $\lfloor |G|/|\widetilde{M}|\rfloor=n$. 
	Moreover, define $\Cl_{trans}(G) \coloneqq \Cl_{imprim}(G)\cup\Cl_{prim}(G)$. We aim to show that, for all $n \ge 1$:
	\begin{enumerate}
		\item $|\Cl_{intrans}(G)| \le c_{intrans} n^{1/ \delta}$, for some absolute constant $c_{intrans}$;
		\item $|\Cl_{trans}(G)|=n^{o(1)}$.
	\end{enumerate}
	Since $\mn_{n}(G)=|\Cl_{intrans}(G)|+|\Cl_{trans}(G)|$, the result will follow.
	
	We first deal with the intransitive case. 
 Since the intransitive maximal subgroups of $G$ are of the form $S_k \times S_l$, with $k+l=r$ and $k \neq l$, we may write $C_{intrans}(G)=\left\{M_{k_{1}},M_{k_{2}},\dots,M_{k_{s}}\right\}$, where the $k_{i}$'s are distinct integers, $1\le k_{i} < r/2$ and each $M_{k_{i}} \cong S_{k_i} \times S_{r-k_i}$ fixes a $k_{i}$-set (set-wise). Assume also that $k_{1}< k_{2}<\hdots< k_{s}$. In this way, $|\Cl_{intrans}(G)|\le k_{s}$.
	Note that $\pi \in G$ fixes a $k_i$-set if and only if $\pi$ belongs to some subgroup of the form $S_{k_i} \times S_{r-k_i}$, that is $\pi \in \widetilde M_{k_i} $. Therefore, the proportion $i(r,k_i)$ of elements of $G$ fixing a $k_i$-set is equal to $|\widetilde{M}_{k_i}|/|G|$.
	By Theorem \ref{EFGTheorem}, there exists an absolute constant $D_{1}'$ such that
	$$D_{1}'k_{s}^{\delta}(1+\log{k_{s}})^{3/2}\le |G|/|\widetilde{M}_{k_{s}}|. $$ 
	It follows that $D_1'k_{s}^{\delta} \leq  2\lfloor |G|/|\widetilde{M}_{k_{s}}| \rfloor =2n$, and hence
	$$|\Cl_{intrans}(G)|\le k_{s} \le \left( \frac{2}{D_{1}'} \right)^{1 / \delta}n^{1 /  \delta}.$$
	Therefore, we may set $c_{intrans} \coloneqq \left(2/{D_{1}'} \right)^{1 / \delta}$.
	Now, we consider $\Cl_{trans}(G)$, which again we may assume non-empty. The imprimitive maximal subgroups of $G$ are of the form $S_k \wr S_{l}$, where $ kl=r, k > 1$ and $l >1$.
	Therefore, $|\Cl_{imprim}(G)| \leq d(r)$, where $d(r)$ denotes the set of positive integer divisors of $r$. Also, Liebeck, Martin and Shalev \cite{LieMartShal} have proved that
	$S_r$ has $r^{o(1)}$ conjugacy classes of primitive subgroups, thus $|\Cl_{prim}(G)| \leq r^{o(1)}$. Finally, by \cite[Section 5]{LuPyb}, $n\ge r^{\alpha}$ for some $\alpha>0$. Since it is well known that $d(r)=r^{o(1)}$, we have $|\Cl_{trans}(G)|=n^{o(1)}$, and in particular there exists an absolute constant $c_{trans}$ such that  $|\Cl_{trans}(G)| \le c_{trans}n^{1/\delta}$ for all $n\geq 1$. The result follows by taking $C_1 \coloneqq c_{intrans}+c_{trans}$.
\end{proof}

\subsection{Classical simple groups}

In dealing with the case when $G$ is a classical simple group, we shall distinguish the maximal subgroups of $G$ using Aschbacher's Theorem \cite{Asch84}; we will use the formulation from \cite{KL}, and we will adopt the notation used there, although for ease of exposition, we will write $\Cl_{9}$ instead of the more commonly used $\mathcal{S}$.
We will need the following preliminary lemma. The proof relies on a series of papers by Fulman and Guralnick in which they prove that the proportion of derangements is bounded away from zero for any simple transitive group, confirming a conjecture of Boston and Shalev. Fulman and Guralnick's main result is stated and used later in this paper (see Theorem \ref{Bost}). Moreover, we will make use of the theory of Shintani descent (see \cite[Chapter 3]{Harp}).

\begin{lemma}\label{PreClassicalCase} Let $G=X_{r}(q)$ be a finite simple classical group of (untwisted) Lie rank $r$, defined over a field $\FF$ of order $q$, and let $m$ be the dimension of the natural module $V$ for $G$. Let $M$ be a maximal subgroup of $G$. Then there exist absolute constants $C_{1}$, $C_{2}$, $C_{4}$, $B'$, $\alpha$ and $\beta$, and a function $f(r,q)$ which tends to $0$ if either $r$ or $q$ is increasing, such that the following holds.
	\begin{enumerate}
		\item[\upshape(i)] If $M$ stabilises a $k$-dimensional subspace of $V$, with $1\le k\le m/2$, then $\lfloor |G|/|\w{M}|\rfloor\ge C_{1}k^{0.005}$.
		\item[\upshape(ii)] If $M$ lies in the Aschbacher class $\Cl_{2}$, then $\lfloor |G|/|\w{M}|\rfloor\ge C_{2}r^{\alpha}$.
		\item[\upshape(iii)] If $M$ stabilises an extension field of $\FF$ of degree $b$, for an odd prime $b$, then $|\w{M}|/|G|\le 1/b+f(r,q)$, where $f(r,q)\coloneqq B'(1 + \log_{q}{r})/q^{r/2-1}$.
		\item[\upshape(iv)] If $M$ lies in one of the Aschbacher classes $\Cl_{i}$ for $4\le i\le 9$, then $\lfloor |G|/|\w{M}|\rfloor\ge C_{4}q^{r/3}$.
	\end{enumerate}
\end{lemma} 

\begin{proof}
	Part (i) follows from \cite[Theorems 2.2, 2.3, 2.4 and 2.5]{FGC2C3}.
More explicitly, from \cite[Theorem 2.2]{FGC2C3} we obtain that if $1 \le k \le m/2$, then the proportion of elements of $\SL_m(q)$ which fix a $k$-space is at most $A/k^{0.005}$ for a universal constant $A$ and, if $H \le \SL_m(q)$ is a maximal subgroup fixing a $k$-dimensional subspace, such probability is exactly equal to $|\widetilde H|/|\SL_m(q)|$. Completely analogous results for the groups $\SU_m(q^{1/2}), \Sp_{m}(q)$ and $\Omega^{\epsilon}_{m}(q)$ follow from \cite{FGC2C3}, Theorems 2.3, 2.4 and 2.5,
	respectively. 
	\paragraph{}
	Part (ii) follows from \cite[Theorem 1.4]{FGC2C3}, which asserts that if $M$ is a $\mathcal{C}_2$-subgroup, then $|\w{M}|/|G| \le A/r^{\delta}$, for some absolute constants $A$ and $\delta$.
	\paragraph{}
	Part (iv) follows from \cite[Lemmas 7.8, 7.9, 7.10, 7.11 and 7.12]{FGC4C9}. More explicitly, Lemma 7.8 asserts that
	if $\mathcal{X}(G)$ denotes the set of maximal subgroups of $G$ contained in $\Cl_i$, for $4\le i\le 8$. 
	Then $$  \frac{|\bigcup_{M \in \mathcal{X}(G)}M|}{|G|} < O(q^{-r/3} ).$$ A completely analogous result for the maximal subgroups of $G$ contained in $\Cl_9$ follows from Lemmas 7.9, 7.10, 7.11 and 7.12.
	\paragraph{}
	Finally, Part (iii) follows by arguing as in the proof of Corollary 5.3 of \cite{FGC2C3}; we repeat the details here for the readers benefit.  Set
	\begin{equation*}
		(H,H_0):=
		\begin{cases}
			(\GL_{m/b}(q^{b}).b,\GL_{m/b}(q^{b})), & \text{if}\ X=\textbf{L} \\
			(\GU_{m/b}(q^{b/2}).b,\GU_{m/b}(q^{b/2})), & \text{if}\ X=\textbf{U}
			\\
			(\Sp_{m/b}(q^{b}).b,\Sp_{m/b}(q^{b})), & \text{if}\ X=\textbf{S}
			\\
			(\SO^{\epsilon}_{m/b}(q^{b}).b,\SO^{\epsilon}_{m/b}(q^{b})), & \text{if}\ X=\textbf{O}^{\epsilon}, \text{with}\ \epsilon\in\left\{+,-,\circ\right\}.
		\end{cases}
	\end{equation*}
	Recall  that $m=r+1$ if $X=$ \textbf{L} or $X=$ \textbf{U}, $m=2r$ if $X=$ \textbf{S} or $X=$ \textbf{O}$^{\pm}$, and $m=2r+1$ if $X=$ \textbf{O}$^{\circ}$. 
	Let $Y_{m}(q) \in \{\GL_{m}(q), \GU_{m}(q^{1/2}), \Sp_{m}(q), \SO^{\epsilon}_{m}(q)\}$ and let $Q_{m}(q) \in \{\SL_{m}(q), \SU_{m}(q^{1/2}),  \Sp_{m}(q), \SO^{\epsilon}_{m}(q)\}$. Also, let $\varphi$ be the generator of the cyclic group of order $b$ in the definition of $H$, and recall that $\varphi$ induces a generalised $q$-Frobenius map on $H_{0}$. By Shintani descent, there is a bijection between $H$-conjugacy classes in the coset $H_{0}{\varphi}^i$ and conjugacy classes in $Y_{m/b}(q^{i})$, for each $0<i<b$, with $i \mid b$. 
	Now, by \cite[Corollary 1.2]{FGC4C9}, there is a constant $c$ (independent of $q$ and $r$) such that $$k(Y_{m/b}(q^i))\le cq^{ri/{b}}.$$ So there are at most $\sum_{i=1}^{\lfloor b/2 \rfloor }cq^{ri/b}$ $H$-conjugacy classes in $H \setminus H_0$. This number  is easily seen to be at most $2cq^{r/2}$. Therefore, there are at most $2cq^{r/2}$ $Y_{m}(q)$-conjugacy classes of $H$ that intersect $H \setminus H_0$, and these are exactly the conjugacy classes of $Y_{m}(q)$ that intersect $H \setminus H_0$.
	By \cite[Theorem 2.1]{FGC2C3}, there is an absolute constant $A$ such that, for all $y \in Y_{m}(q)$:
	\begin{equation}
		\label{ge}
		|C_{Y_m(q)}(y)| \ge \frac{q^r}{A(1+\log_q{r)}}.
	\end{equation}
	This implies that the proportion of elements of $Y_{m}(q)$  which intersect some conjugate of $H \setminus H_0$ is bounded above by $B(1 + \log_{q}{r})/q^{r/2}$ for some universal constant $B$: indeed if $y_1^{Y_m(q)}, \dots, y_t^{Y_m(q)}$ are the distinct conjugacy classes of $Y_m(q)$ which intersect $H \setminus H_0$, then we have just proved that $t \le 2cq^{r/2}$, and using (\ref{ge}) yields:
	\begin{align*}
		& \frac{\bigcup_{i=1}^t|y_i^{Y_{m}(q)}|}{|Y_m(q)|}=
		\frac{\bigcup_{i=1}^t[Y_m(q):C_{Y_m(q)}(y_i)]}{|Y_m(q)|}=\bigcup_{i=1}^t{\frac{1}{|C_{Y_m(q)}(y_i)|}} \\
		& \le t \frac{A(1+\log_q{r})}{q^r} \le 2cq^{r/2} \frac{A(1+\log_q{r})}{q^r}=\frac{B(1+\log_{q}{r})}{q^{r/2}},
	\end{align*}
	where we defined $B \coloneqq 2cA$.
	Thus, the proportion of elements of $Q=Q_{m}(q)$ contained in a conjugate of $H \setminus H_0$ is at most $B'(1 + \log_{q}{r})/q^{r/2-1} \coloneqq f(r,q)$, for an absolute constant $B'$.
	Finally, by the orbit-stabiliser theorem, $|\{H_0^y \mid y \in Q\}|=[Q:N_Q(H_0)]$, and since $[N_Q(H_0):H_0]=b$, the union of the $Q$-conjugates of $H_0$ contains at most $[Q:N_Q(H_0)]|H_0|=|Q|/b$ elements and therefore the proportion of elements of $Q$ contained in a conjugate of $H$ is 
	$$ \frac{|\w{H \cap Q}|}{|Q|} \le \frac{1}{b}+\frac{B'(1+\log_{q}r)}{q^{r/2-1}}.$$
\end{proof}
We also require the following lemma.
\begin{lemma}
	\label{GKS}
	Let $G$ be a classical simple group of (untwisted) Lie rank $r$, defined over a field $\FF$ of order $q$ and let $m$ be the dimension of the natural module $V$ for $G$. Let $\Delta$ be the normalizer of G in the corresponding projective linear group, and let $\rho_i(G)$ be the number of $\Delta$-conjugacy classes of maximal subgroups of $G$ in Aschbacher class $\Cl_i$. Then
	\begin{enumerate}
		\item[\upshape(i)] $\rho_1(G) \le (3/2)m$;
		\item[\upshape(ii)] $\rho_2(G) \le 2d(m)+1$, where $d(m)$ is the number of divisors of $m$;
		\item[\upshape(iii)] $\rho_3(G) \le \pi(m)+2$, where $\pi(m)$ is the number of prime divisors of m;
		\item[\upshape(iv)] $\rho_4(G) \le 2d(m)$;
		\item[\upshape(v)] $\rho_5(G) \le \log \log(q)$;
		\item[\upshape(vi)] $\rho_6(G) \le1 $;
		\item[\upshape(vii)] $\rho_7(G) \le 3\log{m}$;
		\item[\upshape(viii)] $\rho_8(G) \le 4$;
		\item[\upshape(ix)] $\rho_9(G) \le B r^6$, for some absolute constant $B$.
		
	\end{enumerate}
	Furthermore, each such $\Delta$-class splits into at most $m$ $G$-classes. 
	\begin{proof}
		The statements (i)--(viii) are \cite[Lemma 2.1]{GurKantSax}, and their proof follows from \cite[Chapter 4]{KL}; while Part (ix) follows from the proof of \cite[Theorem 6.3]{GurLarTie}. The final sentence follows from inspection of \cite[Tables 3.5.A--Tables 3.5.G]{KL}.
	\end{proof}
\end{lemma}
\begin{proposition}

	\label{ClassicalCase} Let $G$ be a finite simple classical group. Then there exists an absolute constant $C$ such that $\mn_{n}(G)\le Cn^{200}$ for all $n\ge 1$.
	\begin{proof} Fix a positive integer $n$, and for each $1\le i\le 9$, let $\mathcal{C}_{i}(G)$ be a set of representatives for the conjugacy classes of maximal subgroups $M$ of $G$ which lie in the Aschbacher class $\mathcal{C}_{i}$, and satisfy $\lfloor |G|/|\widetilde{M}|\rfloor=n$. Write $G=X_{r}(q)$, where $r$ is the Lie rank of $G$, $q$ is the order of the field $\mathbb F$ over which $G$ is defined, and let $m$ be the dimension of the natural module. Similar to the proof of Proposition \ref{AlternatingCase}, our strategy will be to prove that for each $1\le i\le 3$, there exists an absolute constant $c_{i}$ such that $|\Cl_{i}(G)|\le c_{i}n^{200}$; furthermore, we will show that there exists an absolute constant $c_{4}$ such that $|\bigcup_{i=4}^{9}\Cl_{i}(G)|\le c_{4}n^{200}$. The result will then follow by taking $C:=c_{1}+c_{2}+c_{3}+c_{4}$. 
		\paragraph{}
		If $\mathcal{C}_{1}(G)$ is non-empty, write $\mathcal{C}_{1}(G)=\left\{M_1,M_2,\hdots,M_s\right\}$, where for each $i$, $M_{i}$ fixes a $k_{i}$-dimensional subspace of the natural module for $G$ for some $1\le k_{i}\le m/2$. We may also assume that $k_{1} \le k_{2} \le \hdots \le k_{s}$. By \cite[Tables 3.5.A--3.5.G]{KL}, for a fixed $k_i$, there at most 3 distinct conjugacy classes of maximal subgroups which fix a $k_i$-set.
		So, $|\mathcal{C}_{1}(G)|\le 3k_{s}\le 3(n/C_{1})^{200}$, where the last inequality follows from Lemma \ref{PreClassicalCase} Part (i), and therefore we may take $c_{1}:=3(1/C_{1})^{200}$.
		\paragraph{}
		If $\mathcal{C}_{2}(G)$ is non-empty, then Lemma \ref{PreClassicalCase} Part (ii) implies that there exist absolute constants $C_{2}$ and $\alpha$ such that $n\ge C_{2}r^{\alpha}$. Moreover, by Lemma \ref{GKS}, $|\mathcal{C}_{2}(G)|\le 2d(m)+1$, where $d(m)$ is the number of divisors of $m$. As mentioned in the proof of Proposition \ref{AlternatingCase}, it is well known that $d(m)=m^{o(1)}$. Recalling also the relation between the Lie rank $r$ and the dimension $m$ of the natural module, it follows that $|\mathcal{C}_{2}(G)|=n^{o(1)}$, and hence that the constant $c_{2}$ exists. 
		\paragraph{}
		Now, we consider $\mathcal{C}_{3}(G)$, which again we may assume  non-empty. Write $\mathcal{C}_{3}(G)=\left\{M_1,M_2,\hdots,M_t\right\}$, where $M_i$ stabilizes an extension field of $\FF$ of prime degree $b_{i}\mid m$. Assume also that $b_{1}\le b_{2}\le \hdots\le b_{t}$. Now, it follows from \cite[Tables 3.5.A--3.5.G]{KL} that a prime divisor $b$ of $n$ occurs at most $4$ times among the $b_i$. Thus, at least $\lceil t/4\rceil$ of the $b_{i}$ are distinct. Hence, $|\Cl_{3}(G)|=t\le 4b_{t}$. 
		If $b_t=2$, then $|\Cl_{3}(G)| \le 4$ and the result is clear, 
		so we may assume that $b_{t}$ is odd. 
		Let $C_{3}$ be an absolute constant such that if $\max\left\{q,r\right\}>C_{3}$, then $B'(1 + \log_{q}{r})/q^{r/2-1} \le (2r+1)^{-1}$, where $B'$ is the constant appearing in Lemma \ref{PreClassicalCase}(iii). Then since $(2r+1)^{-1}\le 1/m\le 1/b_t$, Lemma \ref{PreClassicalCase}(iii) implies that if $\max\left\{q,r\right\}>C_{3}$, then $|\Cl_{3}(G)|=t\le 4b_{t}\le 8n$. If $\max\left\{q,r\right\}<C_{3}$, then $\Cl_{3}(G)=O(1)$. Either way, the existence of the constant $c_{3}$ follows.
		
		\paragraph{}
		Finally, by Lemma \ref{PreClassicalCase} Part (iv), there exists an absolute constant $C_{4}$ such that if $\mathcal{C}_{i}(G)$ is non-empty for some $4\le i\le 9$, then $n\ge C_{4}q^{r/3}$. One can now easily deduce, from the upper bounds (iv)--(ix), together with the final statement, in Lemma \ref{GKS}, that $|\bigcup_{i=4}^{9}\Cl_{i}(G)|\le c_4n$ for some absolute constant $c_4$. This completes the proof.
	\end{proof}
	
\end{proposition}
\subsection{Exceptional simple groups of Lie type}
\begin{proposition}\label{ExceptionalCase} Let $G$ be a simple exceptional group of Lie type. For $n\ge 1$, we have $\mn_{n}(G)=n^{o(1)}$ and $\mn_{n}(G)=O(n)$.
	\begin{proof} Write $G=$ $^{\epsilon}X_{r}(q)$, where $r$ is the (untwisted) Lie rank of $G$, and $q$ is the order of the field of definition. By  \cite[Corollary 4]{LieSeitz}, there exists an absolute constant $D$ such that if a maximal subgroup $M$ of $G$ has order larger than $D$, then $M$ falls into at most $O(\log\log{q})$ conjugacy classes of subgroups of $G$. Furthermore, the number of conjugacy classes of maximal subgroups $M$ of $G$ satisfying $|M|\le D$ is $O(1)$, by \cite[Theorem 1.2]{LieMartShal}. Also, it is shown in \cite{FGsurvey} that there exists an absolute constant $C$ such that $|G|/|\widetilde{M}|\ge Cq$ if $M$ is not a maximal subgroup of $G$ which has maximal rank. Finally, the number of conjugacy classes of maximal subgroups of $G$ of maximal rank is $O(1)$, by \cite[Main Theorem and Table 5.1]{LieSaxSeitz}. The proposition follows.\end{proof}
\end{proposition}

\subsection{Proof of Theorem \ref{maineinv}}
In the previous subsections, we have estimated $\mn_n(G)$, when $G$ is a finite non-abelian simple group. Now, with the help of this information, we want to deduce   a key estimate of $\mn_n(G)$, when $G$ is a direct product of simple groups.  
Before we deduce this result, we require some additional notation. Suppose that  $T$ is a nonabelian finite simple group, and fix $\alpha\in \aut(T)$. Define the subgroup $D_{\alpha}\le T\times T$ by $$D_{\alpha}:=\left\{(t,t^{\alpha})\text{ : }t\in T\right\}.$$ Then, it is easy to see that $D^{\alpha}$ and $D^{\beta}$ are conjugate in $G$ if and only if $T\alpha=T\beta$. Moreover, $D_{\alpha}$ is a maximal subgroup of $T\times T$. Finally, we note that, for a fixed $\alpha\in \aut(T)$, the size of the $(T\times T)$-conjugacy class of $(t,t^{\alpha})\in G$ is $[T:C_{T}(t)]^{2}$, i.e. it is the square of the size of the $T$-conjugacy class of $t$.

We are now ready to prove the afore mentioned key estimate. 
\begin{proposition}\label{MaxDirect} Let $G=T_{1}\times \hdots\times T_{k}$ be a direct product of $k \ge 1$ nonabelian finite simple groups. Then there exist absolute constants $c_0$ and $C'$ such that $\mn_{\mathrm{prod},n}(G)\le c_0kn^{200}$ and $\mn_{n}(G)\le C'k^{2}n^{200}$ for all $n\ge 1$, where $\mn_{\mathrm{prod},n}(G)$ is the number of conjugacy classes of maximal subgroups of $G$ of product type satisfying $\lfloor |G|/|\w{M}|\rfloor=n.$ 
	\end{proposition}
	\begin{proof} As we have recalled in Section \ref{secondsec}, the maximal subgroups $M$ of $G$ fall into  two categories. We adopt here the following notation: \begin{enumerate}
			\item \emph{Product type}: $M$ is of the form $M=M_{i}\times \hat{T}_{i}$, where $\hat{T}_{i}:=\prod_{l\neq i} T_{l}$, $1\le i\le k$, and $M_{i}$ is a maximal subgroup of $T_{i}$.

			\item \emph{Diagonal type}\label{diagtype}: $M$ is of the form $M=D_{i,j,\alpha}\times \hat{T}_{i,j}$, where $\hat{T}_{i,j}:=\prod_{l\neq i,j} T_{l}$, $T_{i}\cong T_{j}$, $1\le i<j\le k$, $\alpha\in \aut(T_{i})$ and  $D_{\alpha}\cong D_{i,j,\alpha} \le T_{i}\times T_{j}$.\end{enumerate}
		Now, fix $n$, and let $\Cl_{prod}(G)$ [resp. $\Cl_{diag}(G)$] be a set of representatives for the conjugacy classes of maximal subgroups $M$ of $G$ of product [resp. diagonal] type satisfying $\lfloor |G|/|\w{M}|\rfloor=n$. Then it follows immediately from Propositions \ref{AlternatingCase}, \ref{ClassicalCase} and \ref{ExceptionalCase} that $|\Cl_{prod}(G)|\le C_{prod}kn^{200}$ for some absolute constant $C_{prod}$. So we just need to prove that there exists a constant $C_{diag}$ such that $|\Cl_{diag}(G)|\le C_{diag}k^{2}n^{200}$.
		
		Let $M\in \Cl_{diag}(G)$, and let $T:=T_{i}\cong T_{j}$, and $\alpha\in \aut{(T)}$ be as in the description of diagonal type subgroups at (\ref{diagtype}). 
		Then, using the discussion preceding this proposition, we have 
		\begin{equation}
			\label{diag}
			|\Cl_{diag}(G)|\le k(k-1)|\out{(T)}|.
		\end{equation}
		The result then follows when $T$ is an alternating group, since $|\out(T)| \leq 4$. 
		So, assume that $T$ is a finite simple group of Lie type, of (untwisted) Lie rank $r$, and field of definition of order $q$. 
		Moreover
		$$\widetilde D_\alpha=\{(t^x,t^{\alpha y})\mid t,x,y \in T\}=\{(u,u^{x^{-1}\alpha y})\mid u,x,y \in T\}=\{(u,u^{\alpha z})\mid u,z \in T\}.$$
		Thus
		\begin{align}\label{impnote} |\w{M}|=|\hat{T}_{i,j}||\widetilde D_\alpha|=|\hat{T}_{i,j}|\sum_{u\in T}[T:C_T(u^\alpha)]=|\hat{T}_{i,j}|\sum_{u\in T}[T:C_T(u)].\end{align}

		Since $|C_{T}(t)|\ge (q-1)^{r}/r$ by \cite[Lemma 3.4]{FGsurvey}, 
		$|\w{M}|\leq|\hat{T}_{i,j}| |T|^2/(q-1)^r=|G|r/(q-1)^r$ and $2n=2\lfloor|G|/|\w{M}| \rfloor \geq  |G|/|\w{M}| \geq (q-1)^r/r.$  Since $|\out{(T)}|=O(r\log{q})$ (see for example \cite{boundautG}), the result now follows from (\ref{diag}).   
	\end{proof}
Before proceeding to the proof of Theorem \ref{maineinv}, we note the result of Fulman and Guralnick mentioned in the introduction.

\begin{theorem}[{\cite[Theorem 1.1]{FGC2C3}}]
	\label{Bost}
	Let $G$ be a finite simple group acting faithfully and transitively on a set $X$ of cardinality $n$. With possibly finitely many exceptions, the proportion of derangements in $G$ is at least 0.016. \end{theorem}

Note that any faithful transitive action of a finite simple group $G$ is isomorphic to the action by right multiplication on the set of right cosets of a subgroup $M$ of $G$, and the stabiliser of $Mg$ is $M^g$. Therefore, $\w{M}$ is the set of elements of $G$ having at least one fixed point under this action. Thus, Theorem \ref{Bost} can be rephrased as follows.

\begin{theorem} \label{BostonShalev}
	There is an absolute constant $\alpha_0>0$ such that $|G|/|\w{M}|>1+\alpha_0$ whenever $G$ is a non-abelian finite simple group and $M$ is a subgroup of $G$.
\end{theorem}

We also require the following straightforward lemma.
\begin{lemma}\label{2lemma}
Let $G=T_1\times\hdots\times T_k$ be a direct product of nonabelian finite simple groups, let $M$ be a maximal subgroup of diagonal type, and let $\alpha:=\min_{1\le i\le k}\alpha(T_i)$. Then $|G|/\w{M}|>2$ and $|G|/|\widetilde{M}|\geq \min\{4,\alpha^2\}$.\end{lemma}
\begin{proof}
By \eqref{impnote} in the proof of Proposition \ref{MaxDirect}, and adopting the notation therein, we have
\begin{align}\label{impnote2}|\w{M}|=|\hat{T}_{i,j}|\sum_{u\in T}[T:C_T(u)].
\end{align}
Hence, $|G|/|\widetilde{M}|=|T|^2/\sum_{u\in T}[T:C_T(u)]$. If the centraliser of every element of $T$ has order at least $4$, then $\sum_{u\in T}[T:C_T(u)]\le |T|^2/4$, which gives us what we need. 

So assume that $T$ has an element $x$ with $|C_T(x)|<4$. Then since no nonabelian finite simple group has a Sylow subgroup of order $2$, the element $x$ must have order $3$. It now follows from \cite[Theorem A]{AradChillag} that $T=A_5$. One can check that in this case $\alpha(T)=3/2$, and $|T|^2/\sum_{u\in T}[T:C_T(u)]=3600/91>\alpha(T)^2$. By \eqref{impnote2}, the proof is complete.
\end{proof}

We are finally ready to prove Theorem \ref{maineinv}. \begin{proof}[Proof of Theorem \ref{maineinv}] Let $G$ be a  direct product of $k \ge 1$ nonabelian finite simple groups.
By Proposition \ref{MaxDirect}, there exist absolute constants $c_0,C'$ and $\mu=200$ such that $\mn_{\mathrm{prod},n}(G)\le c_0kn^\mu$ and $\mn_{n}(G)\le C'k^2n^\mu$ for all $n\ge 1$. Let $\alpha:=\min_{1\le i\le k}\alpha(T_i)$, and set 
$$\beta:=\left\lceil \max\left\{\frac{\log(c_0k)}{\log \alpha}, \mu+\log(c_0k),\mu+\frac{\log(C'k^2)}{\log(\min\{4,\alpha^2\})}\right\}\right\rceil.$$ Also, let $\mathrm{Max}_{2}(G)$ be a set of representatives for the conjugacy classes of those maximal subgroups $M$ of $G$ satisfying $|G|/|\w{M}|<2$. By Lemma \ref{2lemma}, all groups in $\mathrm{Max}_{2}(G)$ have product type. Thus, 
 $|\widetilde{M}|/|G|\le 1/\alpha$ for all $M\in\mathrm{Max}_{2}(G)$, and by Proposition \ref{MaxDirect}, we have $|\mathrm{Max}_{2}(G)|\le c_0k$. Furthermore, for $n\geq 2$, Proposition \ref{MaxDirect} and Lemma \ref{2lemma} yield $\mn_n(G)\le f_k(n)n^{\mu}$, where $f_k(n):=c_0k$ if $n<\min\{4,\alpha^2\}$, and $f_k(n):=C'k^2$ if $n\geq \min\{4,\alpha^2\}$. Notice that $\mu+\log{f_k(n)}/\log{n}\le \beta$. Hence 
	\begin{align*} 1-\mathbb{P}_{I}(G,t) &\le \sum_{M\in \rm{Max}_{2}(G)}(|\w{M}|/|G|)^{t}+\sum_{n\ge 2}\mn_{n}(G)/n^{t}\\
		&\le \frac{c_0k}{\alpha^{t}}+\sum_{n\ge 2}\frac{f_k(n)n^{\mu}}{n^{t}}\\
		&\le  \frac{\alpha^{\log{(c_0k)/\log(\alpha)}}}{\alpha^{t}}+\sum_{n\ge 2} \frac{n^{\mu+\log(f_k(n))/\log{n}}}{n^{t}}\\
		&  \leq \frac{\alpha^\beta}{\alpha^t}+\sum_{n\ge 2}\frac{n^\beta}{n^{t}}
		=\frac{1}{\alpha^{t-\beta}}+\sum_{n\ge 2}\frac{1}{n^{t-\beta}}.\end{align*}

	Thus
	\begin{align*}C(G) &= \sum_{t\ge 0}(1-\mathbb{P}_{I}(G,t))\\
		&\le \beta+2+\sum_{t\ge \beta+2}(1-\mathbb{P}_{I}(G,t))\\
		&\le \beta+2+\sum_{u\ge 2}\frac{1}{\alpha^{u}}+\sum_{u\ge 2}\sum_{n\ge 2}\frac{1}{n^{u}}\\
		& =\beta+2+\frac{1}{\alpha(\alpha-1)}+\sum_{n\ge 2}\left (\sum_{u\ge 2} \frac{1}{n^{u}}\right )\\
		&= \beta+2+\frac{1}{\alpha(\alpha-1)}+\sum_{n\ge 2}\frac{1}{n(n-1)}\\
		&=\beta+2+\frac{1}{\alpha(\alpha-1)}+\sum_{n\ge 1}\frac{1}{n(n+1)}\\
		&=\beta+\frac{1}{\alpha(\alpha-1)}+3.
	\end{align*}
		Since $\alpha>1+\alpha_0$ for some absolute constant $\alpha_0$ by Theorem \ref{BostonShalev}, the result follows.
\end{proof}

\section{Direct powers of finite simple groups}

In this section, we show that the estimates given in Theorems \ref{maine} and \ref{maineinv}   are best possible, up to improvements in the additive constants.
To this end, fix a nonabelian finite simple group $T$, and consider the direct power $G=T^k$ of $k$ copies of $T$. Set $l:=l(T)$ to be minimum index of a proper subgroup of $T$. Then $m_l(T^k)\geq lk$. Thus, $\mathcal{M}(T^k)\geq \log{k}/\log{l}+1$. It then follows from \eqref{amp} and Proposition \ref{thm1} that
	\begin{equation}\label{ebound}
		e_1(T^k)=\frac{\log k}{\log l}+c_T(k) \text { where }
		-3 \leq c_T(k)\leq 6.
	\end{equation}

Now, choose $M$ to be a maximal subgroup of $T$ with $|T|/|\widetilde{M}|=\alpha:=\alpha(T)$. Consider the following maximal subgroup of product type in $T^k:$
	$$M_i=T \times \cdots \times M \times \cdots \times T,$$
	where $M$ occupies position $i$.
	For, $i_1 < \cdots < i_r$, define $\widetilde{M}_{i_1, \dots, i_r} \coloneqq \bigcap _{j=1}^{r}\widetilde{M}_{i_j}$, so that
	$\frac{|\widetilde{M}_{i_1, \dots, i_r}|}{|G|}=\alpha^{-r}$.
	Moreover, let $\Omega_t \coloneqq \bigcup_{1 \le i \le k}\widetilde{M}_i^{t}$.
	Using the inclusion-exclusion principle and the fact that a subset $\{g_1, \dots, g_t\}$ fails to invariably generate $G$ if and only if it is contained in $\widetilde{M}$ for some maximal subgroup $M$ of $G$, we have
	\begin{align*}
		& 1-\mathbb{P}_{I}(G,t)  \geq \frac{|\Omega_t|}{|G|^t}  =\frac{\sum_{i}|\widetilde{M}_i|^t-\sum_{i_1 < i_2}{|\widetilde{M}_{i_1,i_2}|^t}+ \dots + (-1)^{k+1}\sum_{i_1 < \dots <i_k}|\widetilde{M}_{i_1, \dots, i_k}|^t}{|G|^t} \\
		&= k \alpha^{-t}-\binom{k}{2}\alpha^{-2t}+ \dots +(-1)^{k+1}\binom{k}{k}\alpha^{-kt}=1- \sum_{j=0}^{t}\binom{k}{j}(-\alpha^{-t})^j
		=1-(1-\alpha^{-t})^k.
	\end{align*}        
	Therefore,
 \begin{align}\label{Cbound}
  C(G)=\sum_{t=0}^{\infty}1-\mathbb{P}_{I}(G,t) \geq \sum_{t=0}^{\infty} 1-(1-\alpha^{-t})^k.   
 \end{align}
		To conclude, fix $\epsilon>0$. We show that there exist absolute constants $C_{\epsilon}$ and $k_{\epsilon}$ such that   
	\begin{equation*}
		C(G) \geq (1-\epsilon)\frac{\log{k}}{\log{\alpha}}+C_{\epsilon}
	\end{equation*}
	whenever $k\geq k_{\epsilon}$ 
	To this end, observe that for fixed $a>0$, $1-(1-a^x)^k$ is a decreasing function for $x \geq 0$. Therefore, if $\mu<1/k$ and $a<1$ then
	\begin{align*}
		\sum_{n=0}^{\infty} 1-(1-a^n)^k & \geq \sum_{n \leq \log_{a}{(1/\mu k)}}1-(1-a^n)^k \\
		& \geq \sum_{n \leq \log_{a}{(1/\mu k)}}1-(1-a^{\log_{a}{1/\mu k}})^k \\ 
		&= \sum_{n \leq \log_{a}{(1/\mu k)}} 1-\left(1 -\frac{1}{\mu k} \right)^k \\
		& =\left(1-\left( 1- \frac{1}{\mu k}\right)^k \right)\log_{1/a}{(\mu k)} \geq \left(1- \frac{1}{e^{1/\mu}} \right)\log_{1/a}{(\mu k)},
	\end{align*}
	where, to obtain the last inequality, we simply used the fact that $\left( 1-1/\mu k\right)^k \le 1/e^{1/\mu}$. Now choose $\mu<\min\{1/k,1/\log_e(1/\epsilon)\}$; set $a:=1/\alpha$; and set $C_{\epsilon}:=(1-\epsilon)\log{\mu}/\log{(1+\alpha_0)}$, where $\alpha_0$ is as in Theorem \ref{BostonShalev}. Applying \eqref{Cbound} then yields  
\begin{equation}\label{Cbound2}
	C(T^k)\geq \left(1-\frac{1}{e^{1/\mu}}\right)\frac{\log k}{\log \alpha}+\frac{\log{\mu}}{\log{\alpha}}\geq \left(1-\epsilon\right)\frac{\log k}{\log \alpha}+C_{\epsilon}.
\end{equation}
We remark that if $T=A_5$, then $\alpha=3/2$. Applying the first inequality in \eqref{Cbound2} in this case with $\mu:=1$ then yields $C(A_5^k)\geq (1-1/e)\log{k}/\log{\frac{3}{2}}$.  
Therefore, by \eqref{ebound},
$$C(A_5^k)-e_1(A_5^k)\geq \log k\left(\frac{1}{\log \frac{3}{2}}\left(1-\frac{1}{e}\right)-\frac{1}{\log 5}\right)-c_{A_5}(k).
$$
Since $\frac{1}{\log \frac{3}{2}}\left(1-\frac{1}{e}\right)-\frac{1}{\log 5}>0$ this shows that the difference between the two
invariants can be an arbitrarily large number.



\end{document}